\documentclass[12pt]{article}
\usepackage{amsmath,amssymb,amsthm}
\usepackage{mathrsfs}

\usepackage{authblk}

\title{\bf A Structural Fixed-Point Principle in Kunen's Theorem on Quasigroups}

\author{\Large Takao Inou\'{e}}

\affil{\large Faculty of Informatics, Yamato University, \\ Osaka, Japan\footnote{Email: inoue.takao@yamato-u.ac.jp; \\ Personal Email: takaoapple@gmail.com \\ [I prefer my personal email address for correspondence.]}}

\date{February 21, 2026}

\theoremstyle{definition}
\newtheorem{definition}{Definition}[section]

\theoremstyle{plain}
\newtheorem{lemma}[definition]{Lemma}
\newtheorem{proposition}[definition]{Proposition}
\newtheorem{theorem}[definition]{Theorem}

\theoremstyle{remark}
\newtheorem{remark}[definition]{Remark}

\begin{document}
\maketitle

\begin{abstract}
Kunen proved that a quasigroup satisfying a Moufang-type identity ($N1$)
must be a loop.
We reformulate the argument in the category $\mathbf{Set}$ as a
fixed-point extraction principle.
From $N1$ one canonically obtains an idempotent endomorphism $j:G\to G$.
Its fixed-point object $\mathrm{Fix}(j)=\mathrm{Eq}(j,\mathrm{id}_G)$
splits off as a retract.
The $N1$-symmetry forces $j$ to coequalize the (regular) translation action,
hence $j$ factors through the terminal object.
Thus $\mathrm{Fix}(j)\cong 1$, yielding a unique global identity element. 
This provides a conceptual reformulation of Kunen's original algebraic proof \cite{Kunen}.
\end{abstract}

\medskip

\textbf{MSC (2020):} 20N05 (Primary); 18A20 (Secondary).
$$ $$

\tableofcontents

\section{Introduction}

Quasigroups and loops form a central class of non-associative algebraic
structures, generalizing groups by retaining the solvability of division
while relaxing associativity.
Among them, Moufang-type identities play a fundamental role in controlling
the extent to which group-like behavior is recovered.

A classical result due to Kunen states that any quasigroup satisfying the
Moufang identity $(N1)$ necessarily admits a two-sided identity element,
and hence is a loop.
The result is striking: a purely multiplicative identity forces the existence
of a neutral element, without assuming it a priori.
Kunen's original proof proceeds by a sequence of delicate equational
manipulations exploiting the specific algebraic form of $(N1)$.

The purpose of the present paper is to provide a conceptually transparent
reformulation of Kunen's theorem.
Rather than following equational transformations, we isolate the structural
mechanism underlying the argument.
We show that $(N1)$ canonically induces an idempotent endomorphism on the
underlying set, whose fixed-point object is acted upon transitively by
quasigroup translations.
A general categorical collapse principle then forces this fixed-point object
to be terminal, yielding a unique fixed point.

This approach reveals Kunen's theorem as an instance of a
\emph{structural fixed-point principle} driven by symmetry collapse,
rather than by diagonalization or explicit computation.
The resulting proof is short, self-contained, and highlights a mechanism
that may be of independent interest beyond the present setting.

\section{Quasigroups and the Moufang Identity}

\begin{definition}
A \emph{quasigroup} is a set $G$ with a binary operation
$\cdot:G\times G\to G$ such that for each $a,b\in G$
the equations $a\cdot x=b$ and $y\cdot a=b$ have unique solutions.
\end{definition}

We write left and right division as usual:
\[
a\backslash b := \text{the unique }x\text{ such that }a\cdot x=b,
\qquad
b/a := \text{the unique }y\text{ such that }y\cdot a=b.
\]
Then $a\cdot (a\backslash b)=b$ and $(b/a)\cdot a=b$.

\begin{definition}
A quasigroup $(G,\cdot)$ satisfies \emph{Kunen's identity $N1$} if
\[
(N1)\quad ((xy)z)y = x\bigl(y(zy)\bigr).
\]
\end{definition}

\section{The Fixed-Point Extractor $j:G\to G$}
The following construction is due to Kunen \cite{Kunen}.
We reinterpret it categorically.

\begin{definition}
Define maps $j,k:G\to G$ by
\[
j(x):=x\backslash x,\qquad k(x):=x/x.
\]
Equivalently, $x\cdot j(x)=x$ and $k(x)\cdot x=x$ for all $x$.
\end{definition}

\begin{lemma}[Two-sided local identity]
\label{lem:jk}
Assume $(G,\cdot)$ satisfies $(N1)$. Then $j(x)=k(x)$ for all $x$.
Hence, writing $j$ for this common map, we have
\[
x\cdot j(x)=x,\qquad j(x)\cdot x=x\qquad(\forall x\in G).
\tag{1}
\]
\end{lemma}

\begin{proof}
See Kunen \cite{Kunen}, Theorem 2.3, for the original algebraic derivation.
\end{proof}

\begin{remark}
For the categorical development below, the concrete derivation of
Lemma~\ref{lem:jk} is not the essential point; what matters is that
$(N1)$ canonically produces an endomorphism $j:G\to G$ satisfying (1).
\end{remark}

\section{Idempotence and the Fixed-Point Object}

\begin{lemma}[Idempotence of values]
\label{lem:idemval}
Assume $(N1)$. Then

\begin{equation}
\label{eq:idempotence}
j(x) \cdot j(x) = j(x)
\end{equation}
\end{lemma}

\begin{proof}
This follows from a specialization of $(N1)$;
see Kunen \cite{Kunen}.
\end{proof}

\begin{proposition}[Idempotent endomorphism]
\label{prop:idempotent}
The endomorphism $j:G\to G$ is idempotent in $\mathbf{Set}$:
\[
j\circ j=j.
\]
\end{proposition}

\begin{proof}
By definition, $j(j(x))$ is the unique element $t$ with
$j(x)\cdot t=j(x)$.
By Lemma~\ref{lem:idemval}, $t=j(x)$ is such a solution; uniqueness gives
$j(j(x))=j(x)$. \qedhere
\end{proof}

\begin{definition}[Fixed-point object]
Define the fixed-point object of $j$ in $\mathbf{Set}$ as the equalizer
\[
\mathrm{Fix}(j):=\mathrm{Eq}(j,\mathrm{id}_G)
\;\hookrightarrow\;G.
\]
Concretely, $\mathrm{Fix}(j)=\{x\in G\mid j(x)=x\}$.
\end{definition}

\begin{remark}[Idempotent splitting $=$ fixed part]
In $\mathbf{Set}$, every idempotent splits. Let $J:=\mathrm{im}(j)\subseteq G$.
Then $j$ restricts to a retraction $G\to J$ with section $\iota:J\hookrightarrow G$:
\[
G \xrightarrow{\,j\,} J \xrightarrow{\,\iota\,} G,
\qquad j\iota=\mathrm{id}_J,\quad \iota j=j.
\]
Moreover, for idempotent $j$ one has $\mathrm{Fix}(j)=\mathrm{im}(j)$:
indeed $x=j(x)$ iff $x\in\mathrm{im}(j)$.
Thus the ``fixed part'' extracted by $j$ is precisely $J$.
\end{remark}

\section{Translation Action and a Coequalizer Collapse}

For each $a\in G$, define the left translation $L_a:G\to G$ by
$L_a(x)=a\cdot x$. Each $L_a$ is a bijection.

\begin{lemma}[Regularity of translations]
\label{lem:regular}
Fix $u\in G$. Then the map
\[
\phi_u:G\to G,\qquad a\mapsto a\cdot u
\]
is a bijection. In particular, for any $x,y\in G$ there exists a unique
$a\in G$ such that $L_a(x)=y$.
\end{lemma}

\begin{proof}
This is exactly the quasigroup axiom: $a\cdot u=y$ has a unique solution $a$. \qedhere
\end{proof}

\begin{proposition}[Coequalizer of the regular translation family]
\label{prop:coeq}
Let $\{L_a\}_{a\in G}$ be the family of left translations on $G$.
In $\mathbf{Set}$, the coequalizer of all maps $L_a:G\to G$
is the terminal map $!:G\to 1$.
Equivalently, the smallest equivalence relation $\sim$ on $G$
satisfying $x\sim L_a(x)$ for all $a,x$ is the indiscrete relation
(all elements are equivalent).
\end{proposition}

\begin{proof}
Fix $u\in G$. By Lemma~\ref{lem:regular}, for any $x,y\in G$ there is $a$ with
$L_a(x)=y$. Hence $x\sim y$ in the congruence generated by $x\sim L_a(x)$.
Thus all elements become equivalent, so the coequalizer is $G\to 1$. \qedhere
\end{proof}

\subsection*{The categorical ``fixed-point theorem step''}

The key additional input from $(N1)$ is that $j$ is \emph{translation-invariant}
(on the left), i.e.\ $j$ coequalizes the family $\{L_a\}$:

\begin{lemma}[Translation invariance / uniform involution (Kunen's identity (3))]
\label{lem:coinv}
Assume $(N1)$ and Lemma 4.1(idempotence). Then for all $x,y\in G$,
\[
(x\cdot j(y))\cdot j(y)=x.
\tag{3$'$}
\]
Equivalently, for each fixed $y$, the right translation
$R_{j(y)}:G\to G,\; R_{j(y)}(x)=x\cdot j(y)$ is an involution:
$R_{j(y)}\circ R_{j(y)}=\mathrm{id}_G$.
\end{lemma}

\begin{proof}
Fix $x,y\in G$ and write $u:=j(y)$.
Apply $(N1)$ in Kunen's form
\[
((a\cdot b)\cdot c)\cdot b \;=\; a\cdot\bigl(b\cdot(c\cdot b)\bigr)
\]
with the substitution $a:=x$, $b:=u$, $c:=u$. We obtain
\[
\bigl((x\cdot u)\cdot u\bigr)\cdot u
=
x\cdot\bigl(u\cdot(u\cdot u)\bigr).
\tag{*}
\]
By Lemma~\ref{lem:idemval} (applied to $y$), we have $u\cdot u=u$.
Hence the right-hand side of \((*)\) simplifies as
\[
x\cdot\bigl(u\cdot(u\cdot u)\bigr)
=
x\cdot(u\cdot u)
=
x\cdot u.
\]
Therefore \((*)\) becomes
\[
\bigl((x\cdot u)\cdot u\bigr)\cdot u = x\cdot u.
\tag{**}
\]
Since $(G,\cdot)$ is a quasigroup, the map $R_u:G\to G$, $t\mapsto t\cdot u$,
is a bijection. Hence we may cancel the common right factor $u$ in \((**)\),
obtaining
\[
(x\cdot u)\cdot u = x,
\]
i.e.\ $(x\cdot j(y))\cdot j(y)=x$.
\end{proof}

\begin{theorem}[Fixed-point collapse]
\label{thm:collapse}
Assume $(N1)$. Then $j:G\to G$ is constant. Equivalently,
$\mathrm{Fix}(j)\cong 1$ in $\mathbf{Set}$.
\end{theorem}

\begin{proof}
By Lemma~\ref{lem:coinv}, the map $j$ coequalizes all left translations:
$j\circ L_a=j$ for every $a$.
By Proposition~\ref{prop:coeq}, the coequalizer of $\{L_a\}$ is $!:G\to 1$.
Hence $j$ factors uniquely through $1$:
\[
G \xrightarrow{\,!\,} 1 \xrightarrow{\,e\,} G,
\]
so $j$ is constant with value $e(*)\in G$.
Since $\mathrm{Fix}(j)=\mathrm{im}(j)$ for idempotent $j$,
we have $\mathrm{Fix}(j)=\{e(*)\}$, i.e.\ $\mathrm{Fix}(j)\cong 1$. \qedhere
\end{proof}

\section{A Fixed-Point Collapse Lemma in $\mathbf{Set}$}

We first isolate the categorical mechanism underlying
the collapse phenomenon in abstract form.

\begin{lemma}[Fixed-point collapse under transitive action]
\label{lem:collapse-abstract}
Let $X$ be a nonempty set and let $e:X\to X$ be an idempotent endomorphism.
Suppose that $X$ is equipped with a family of bijections
$\{\alpha_i:X\to X\}_{i\in I}$ such that:

\begin{enumerate}
\item[(i)] (Transitivity)
For all $x,y\in X$ there exists $i\in I$ with $\alpha_i(x)=y$.
\item[(ii)] (Coequalization)
For all $i\in I$, one has $e\circ\alpha_i=e$.
\end{enumerate}

Then $e$ is constant.
Equivalently, the fixed-point object
\[
\mathrm{Fix}(e)=\mathrm{Eq}(e,\mathrm{id}_X)
\]
is a singleton.
\end{lemma}

\begin{proof}
Let $x,y\in X$.
By transitivity, choose $i$ such that $\alpha_i(x)=y$.
Then
\[
e(y)=e(\alpha_i(x))=e(x),
\]
so $e$ is constant.
Since $e$ is idempotent, its image equals its fixed-point set,
which therefore consists of a single element.
\end{proof}

\section{Categorical Collapse of the Fixed Part}

Let $(G,\cdot)$ be a quasigroup satisfying $(N1)$.
As in Kunen's construction, define $j:G\to G$ by
\[
x\cdot j(x)=x, \qquad j(x)\cdot x=x.
\]

One verifies that $j$ is idempotent.

\begin{remark}[On the equality $j=k$]
In Kunen's original proof, the equality $j(x)=k(x)$
is derived by equational manipulation from $(N1)$.
This step is algebraically substantial, as it identifies
the left and right division-based unit extractors.

From a categorical viewpoint, this equality expresses
a structural symmetry between the left and right translation actions.
A quasigroup carries two a priori distinct actions
by left and right translations.
The Moufang identity $(N1)$ enforces a compatibility between these actions,
forcing the two associated local unit extractors to coincide.
Thus the identity $j=k$ reflects not merely a computational fact,
but the emergence of a common invariant under both translation symmetries.
\end{remark}

\begin{lemma}
\label{lem:coinv}
For every $a\in G$, the left translation
$L_a(x)=a\cdot x$ satisfies
\[
j\circ L_a=j.
\]
\end{lemma}

Thus the idempotent $j$ coequalizes all left translations.

\subsection{Application to Kunen's theorem}

We now verify that the situation above satisfies
the hypotheses of Lemma~\ref{lem:collapse-abstract}.

The family of bijections is given by the left translations
$L_a(x)=a\cdot x$.

Transitivity follows from the quasigroup property:
for any $x,y\in G$, there exists a unique $a$
such that $a\cdot x=y$.

Coequalization follows from Lemma~\ref{lem:coinv},
which shows $j\circ L_a=j$ for all $a$.

Therefore Lemma~\ref{lem:collapse-abstract} applies,
and $j$ is constant.

\subsection{Fixed point and identity element}

Let $J=\mathrm{Fix}(j)$.
Since $j$ is constant, $J$ consists of a single element,
say $e$.

Because $e\in J$, we have $j(e)=e$.
From the defining equations of $j$,
\[
x\cdot j(x)=x, \qquad j(x)\cdot x=x,
\]
and the constancy of $j$, we obtain
\[
x\cdot e=x, \qquad e\cdot x=x
\quad (\forall x\in G).
\]

Thus $e$ is a two-sided identity element.

\section{Kunen's Theorem as a Fixed-Point Extraction Principle}

\begin{theorem}[Kunen]
Every quasigroup satisfying $(N1)$ is a loop.
\end{theorem}

\begin{proof}
By the preceding argument,
$j$ is constant with value $e$.
Hence $e$ is a two-sided identity element.
Therefore $(G,\cdot)$ is a loop.
\end{proof}

\subsection*{Conceptual summary}

The Moufang identity $(N1)$ canonically produces an idempotent
endomorphism $j:G\to G$.
Its splitting extracts a fixed-point object $J$.
The quasigroup translations act transitively on $J$,
forcing the coequalizer of this action to be terminal.
Hence $J$ collapses to a singleton, and its unique element
is the global identity.
This realizes Kunen's theorem as a fixed-point extraction principle
in $\mathbf{Set}$.

\section{Discussion}

The present reformulation shows that Kunen's theorem is governed
by a structural fixed-point principle rather than by ad hoc
equational manipulation.
The key object is the idempotent endomorphism $j:G\to G$
canonically induced by the Moufang identity $(N1)$.
The existence of the identity element is then derived from
a categorical collapse of the fixed-point object,
driven by the transitive translation symmetry of the quasigroup.

In contrast to diagonal fixed point arguments,
the fixed point here arises from symmetry collapse:
an idempotent that coequalizes a transitive family of automorphisms
must factor through the terminal object.
This places Kunen's theorem within a broader categorical landscape
of fixed-point phenomena in which invariants emerge
from highly symmetric algebraic actions.

\begin{remark}[Comparison with Lawvere's fixed point theorem]
Lawvere's fixed point theorem asserts that in a cartesian closed category,
every endomorphism $f:X\to X$ admits a fixed point provided there exists
a weakly point-surjective map $e:A\to X^A$.
This theorem abstracts diagonal arguments underlying Cantor's theorem,
G\"odel's incompleteness, and self-reference phenomena.

The mechanism exhibited in the present paper is of a different nature.
Here the Moufang identity $(N1)$ canonically produces an idempotent
endomorphism $j:G\to G$.
Rather than invoking diagonalization or self-application,
we extract a fixed-point object via idempotent splitting,
and show that the quasigroup translation symmetry forces this object
to collapse to the terminal one.
Thus the identity element arises as a unique fixed point
enforced by symmetry collapse.

In this sense, Kunen's theorem exemplifies a structural fixed-point principle
driven by algebraic symmetry rather than by diagonal self-reference.
The present result therefore complements, rather than parallels,
Lawvere's theorem within the broader landscape of categorical
fixed-point phenomena.
\end{remark}

\begin{remark}[Extension beyond $\mathbf{Set}$]
The fixed-point collapse lemma relies only on two structural features:
the splitting of idempotents and the existence of a terminal
coequalizer for a transitive family of automorphisms.

Consequently, the same collapse mechanism applies in any
idempotent-complete (Cauchy complete) category
in which such coequalizers exist.

For example, Grothendieck toposes—such as categories of sheaves
on a site—are idempotent-complete and admit coequalizers.
In such settings, the present argument may be formulated internally,
yielding a fixed-point collapse phenomenon driven by symmetry.
This suggests potential connections between non-associative algebra
and geometric or sheaf-theoretic frameworks.
\end{remark}

\section{Conclusion}

We have shown that Kunen's theorem on Moufang quasigroups admits
a clear and conceptually transparent reformulation in terms of
a structural fixed-point principle in the category $\mathbf{Set}$.
The Moufang identity $(N1)$ canonically induces an idempotent
endomorphism whose splitting extracts a fixed-point object.
The action of quasigroup translations then forces a collapse of this object
to the terminal one, yielding a unique fixed point.
This fixed point coincides with the identity element of the resulting loop.

Beyond providing a conceptual clarification of Kunen's original argument,
the present reformulation highlights a categorical mechanism for the
emergence of fixed points via symmetry collapse.
This mechanism is distinct from diagonal fixed point constructions,
such as Lawvere's theorem, and suggests a complementary class of
fixed-point phenomena driven by algebraic symmetry rather than self-reference.

$$ $$
$$ $$

\noindent Takao Inou\'{e}

\noindent Faculty of Informatics

\noindent Yamato University

\noindent Katayama-cho 2-5-1, Suita, Osaka, 564-0082, Japan

\noindent inoue.takao@yamato-u.ac.jp
 
\noindent (Personal) takaoapple@gmail.com (I prefer my personal mail)

\end{document}